\documentclass[12pt,a4paper]{amsart}

\usepackage{amsmath,amsfonts,amssymb,amsthm,mathrsfs}

\usepackage{color}

\newcommand{\R}{\mathbb{R}}

\newcommand{\un}{\mathbf{1}\!\!{\rm I}} 
\newcommand{\be}{\begin{equation}} 
\newcommand{\ee}{\end{equation}}
\newcommand{\bea}{\begin{eqnarray}} 
\newcommand{\eea}{\end{eqnarray}}
\newcommand{\bean}{\begin{eqnarray*}} 
\newcommand{\eean}{\end{eqnarray*}}
\newcommand{\rf}[1]{(\ref {#1})}

\def\dx{\,{\rm d}x}
\def\dy{\,{\rm d}y}

\def\dr{\,{\rm d}r}
\def\dta{\,{\rm d}\tau}

\def\dl{\,{\rm d}\lambda}

\def\s{\sigma}
\def\p{\partial}

\def\g{\gamma}

\def\r{\varrho}
\def\xn{|\!|\!|}

\def\mn{|\!\!|}

\def\mn2{|\!\!|_{M^{d(p-1)/2}}}

\def\a{\alpha}
\def\J{\widehat J}
\def\ap{\frac{\a}{p-1}}
\def\ts{\tilde s}

\newtheorem{theorem}{Theorem}
\newtheorem{proposition}[theorem]{Proposition}

\newtheorem{corollary}[theorem]{Corollary}

\theoremstyle{definition}

\theoremstyle{remark}
\newtheorem{remark}[theorem]{Remark}

\author[P. Biler]{Piotr Biler}
\address{\small Instytut Matematyczny, Uniwersytet Wroc\l awski,
 pl. Grunwaldzki 2/4, \hbox{50-384} Wroc\-\l aw, Poland}
\email{Piotr.Biler@math.uni.wroc.pl}

\title[Nonlinear nonlocal heat equation]{Blowup of solutions for \\ nonlinear  nonlocal heat equations}

\begin{document}

\begin{abstract} 
Blowup analysis for solutions of  a general evolution equation with nonlocal diffusion and localized source is performed. 
By comparison with recent results on global-in-time solutions, a dichotomy result is obtained. 
\end{abstract}



\date{\today}

\thanks{The author, partially  supported by the NCN grant {2016/23/B/ST1/00434}, thanks  Philippe Souplet and Miko{\l}aj Sier\.z\c ega for many  interesting conversations.} 

\maketitle

\baselineskip=20.5pt
\section{Introduction}\label{intro}

We consider here nonnegative solutions $u=u(x,t)$ of the Cauchy problem 
\bea
u_t&=&J\ast u-u+F(u), \ \ \ x\in\R^d,\ t>0,\label{eq}\\
u(x,0)&=&u_0(x)\ge 0,\label{in}
\eea 
with the linear nonlocal diffusion operator defined by  a nonnegative radially symmetric function $J$ satisfying $\int_{\R^d}J(x)\dx=1$, and with the  nonlinearity (a localized source) defined by a convex function $F:[0,\infty)\to[0,\infty)$, $F\in C^1$, $F(0)=0$ and satisfying 
\be 
\int^\infty\frac{{\rm d}u}{F(u)}<\infty.\label{F}
\ee
Similar evolution equations extending the classical nonlinear heat equation in \cite{QS}
\be
u_t=\Delta u+|u|^{p-1}u\label{nlh}
\ee
and the equations with  nonlocal diffusion operators defined by fractional powers of Laplacian 
\bea
u_t&=&-(-\Delta )^{\a/2}u+|u|^{p-1}u,\ \ x\in {\mathbb R}^d,\ t>0,\label{nlh-a}\\ 
u(x,0)&=&u_0(x),\label{ini} 
\eea
and even more general nonlinearities have been studied in, e.g., \cite{CCR} (the linear case) and  \cite{A,GM-Q,Su} (the nonlinear case). 

Equations of the type \rf{eq} are related to the differential and integrodifferential equations \rf{nlh} and \rf{nlh-a} by their long time asymptotic behavior determined by the linear equations \rf{lin}--\rf{lin-a}, and studied in, e.g., \cite{A,CCR,GM-Q}. 

The main result of this paper is  Theorem \ref{blow} 
 on local-in-time solutions that cannot be continued to global in time ones.

\subsection*{Notation}
The homogeneous Morrey spaces $M^s_q(\R^d)$ modeled on the Lebesgue space $L^q(\R^d)$, $q\ge 1$, are defined for $u\in L^q_{\rm loc}(\R^d)$ and  $1\le q\le s<\infty$, by their norms 
\be
|\!\!| u|\!\!|_{M^s_q}\equiv \left(\sup_{R>0,\, x\in\R^d}R^{d(q/s-1)} 
\int_{\{|y-x|<R\}}|u(y)|^q\dy\right)^{1/q}
=\sup_{R>0,\, x\in\R^d}R^{d(1/s-1/q)}\left\|\un_{B(x,R)}u\right\|_{q}<\infty,\label{hMor}
\ee 
with the convention $M^s_1(\R^d)=M^s(\R^d)$.  

The asymptotic relation $f\approx g$ means that $\lim_{s\to\infty}\frac{f(s)}{g(s)}=1$ 
and $f\asymp g$ is used whenever $\lim_{s\to\infty}\frac{f(s)}{g(s)}\in (0,\infty)$. 
\medskip

\section{Blowup for a general nonlinear source, nonlocal diffusion  model}\label{bl1}

Our aim in this paper is to give a simple proof of blowup of solutions for the Cauchy problem \rf{eq}--\rf{in} based on the classical idea of Fujita in \cite{Fu}. 
We believe that this proof is simpler than monotonicity arguments given in \cite{A}. 
Moreover, this argument applies to a  class of initial data much larger than in \cite[Theorem 2.4]{A}, gives explicit and rather general sufficient conditions on functions $u_0$ in \rf{in} in order to solutions of \rf{eq}--\rf{in} blow up in a~finite time, as well as estimates on the blowup time, cf. also \cite{Su}. 

The linear nonlocal diffusion operator ${\mathcal A}:L^1(\R^d)\to L^1(\R^d)$ defined by 
\be
{\mathcal A}u(x)=(J\ast u)(x)-u(x)\label{A}
\ee
generates the semigroup of linear convolution operators ${\rm e}^{t{\mathcal A}}$ with kernels $k_t\in L^1(\R^d)$, normalized so that $\int_{\R^d} k_t(x)\dx=1$,   which are defined by the inverse Fourier transform ${\mathcal F}^{-1}$ on $\R^d$
\be
k_t(x)={\mathcal F}^{-1}\left({\rm e}^{t(\widehat J(\xi)-1)}\right)(x).\label{kernel}
\ee 
Typical and the most interesting examples of functions $J$ are those with their Fourier transforms $\J$ satisfying  
\be
\J(\xi)=1-A|\xi|^\a+o(|\xi|^\a),\ \ {\rm as }\ \ \xi\to 0,\label{FJ}
\ee
with $A>0$, $\a=2$ (corresponding to, e.g., the case of smooth, compactly supported functions $J$), and those with $\a\in(0,2)$, cf. \cite{CCR}. As it was studied in \cite{CCR}, the long time asymptotics of solutions of the linear Cauchy problem 
\be 
v_t={\mathcal A}v,\ \ v(.,0)=v_0,\nonumber
\ee 
 is then determined either by that of the classical heat equation 
 \be 
 z_t=\Delta z\ \ \ {\rm if}\ \ \ \a=2,\label{lin}
 \ee
 or by the fractional heat equation 
 \be
 z_t=-(-\Delta)^{\a/2}z\ \ \ {\rm in\ the\ case}\ \ \ \a\in(0,2),\label{lin-a}
 \ee
  see \rf{approx} below   for a precise statement from  \cite[Theorem 1]{CCR}. 
Other examples of functions $J$ with $\J$ like $\J(\xi)\approx 1+A|\xi|^2\log|\xi|$ as $ |\xi|\to 0$, are in \cite[Th. 5.1]{CCR} (and then $J(x)\approx\frac{c}{|x|^{d+2}}$ as $|x|\to\infty$). 
 These are examples of operators $\mathcal A$ whose kernels have  ``heavy tails'': $J(x)\approx\frac{c}{|x|^n}$ as $|x|\to\infty$ with some $n>d$, and they are discussed in \cite{A} and \cite{CCR}. 
 For them,  if $n\in(d,d+2)$, then $\a=n-d\in(0,2)$ holds. 
Their semigroup kernels have also heavy tails unlike the Gauss-Weierstrass kernel of the heat semigroup for $\a=2$. 

\begin{theorem}\label{blow}
Suppose that $u_0\ge 0$ satisfies for some $T>0$ the condition 
\be
\frac{W_T(0)}{h^{-1}(T)}>1\label{blo}
\ee
with the moment $W_T(t)$ defined for $t\in[0,T)$ as 
\be
W_T(t)=\int_{\R^d}k_{T-t}(x)u(x,t)\dx.\label{moment}
\ee
Here the decreasing function $h(w)=\int_w^\infty\frac{{\rm d}u}{F(u)}$ satisfies  $h(0)=\infty$, $h(\infty)=0$, by assumption  \rf{F} on the convex function $F$. 
Then, any local in time classical 
solution $u=u(x,t)$ of the Cauchy problem \rf{eq}--\rf{in}  cannot be continued beyond $t=T$.  
\end{theorem}

\begin{proof}
By definition \rf{moment}, we have $W_T(t)=k_{T-t}\ast u(.,t)(0)$ and, of course, $z(.,t)=k_{T-t}$ solves the backward diffusion equation
\be 
z_t=-{\mathcal A}z,\ \ \ z(.,T)=\delta_0\nonumber. 
\ee 
Let us compute the time derivative of $W_T(t)$
\bea
\frac{{\rm d}}{{\rm d}t}W_T(t)&=& \int k_{T-t}(x)\frac{\p}{\p t}u(x,t)\dx + \int \frac{\p}{\p t}k_{T-t}(x)u(x,t)\dx\nonumber\\
&=&\int k_{T-t}(x)({\mathcal A}u)(x,t)\dx+\int k_{T-t}(x)F(u(x,t))\dx +\int (-{\mathcal A}k_{T-t}(x))u(x,t)\dx\nonumber\\
&=&\int k_{T-t}(x)F(u(x,t))\dx\nonumber\\
&\ge& F(W_T(t))\label{der}
\eea
by the symmetry property of the semigroup,  and the Jensen inequality in the last line. 
Integrating this from $0$ to $t$  and passing to the limit  $t\nearrow T$, we obtain 
$$
h(W_T(0))-h(W_T(t))\ge t.
$$
If initially $W_T(0)>h^{-1}(T)$ holds, then taking into account the property $\lim_{w\searrow 0}h(w)=\infty$, we arrive at 
$$\lim_{t\nearrow T}W_T(t)=\infty.$$ 
Finally, if $\lim_{t\nearrow T}W_T(t)=\infty$ then $\limsup_{t\nearrow T,\, x\in\R^d}u(x,t)=\infty$ which means that $u$ blows up not later than at $t=T$. 
\end{proof}

\begin{remark} \label{Besov-cond}
In some particular cases, under more restrictive assumption than condition \rf{F}, the sufficient condition for blowup of solutions of \rf{eq}--\rf{in} can be described in a more explicit way than condition 
\be
\sup_{T>0}\frac{W_T(0)}{h^{-1}(T)}>1\label{blT}
\ee
in Theorem \ref{blow}. 

For the nonlinear heat equation \rf{nlh} the condition \rf{blo} implies 
\be
\sup_{T>0}T^{\frac{1}{p-1}}\left\| {\rm e}^{T\Delta}u_0\right\|_\infty>  \left(\frac{1}{p-1}\right)^{\frac{1}{p-1}},\label{Besov-nlh}
\ee 
a sufficient condition for the blowup of \rf{nlh} which was derived in \cite{Fu},  and has been analyzed in a~recent paper \cite{B-bl}. 
\end{remark}

We have, in this direction, the following 
\begin{proposition}\label{particular}
(i) If the infinitesimal generator $\mathcal A$ of the semigroup ${\rm e}^{t{\mathcal A}}$ satisfies \rf{FJ} for an $\a\in(0,2]$, $F(u)=cu^p$ with some $p>1$ and $c>0$, $u_0\in L^1(\R^d)$, $\widehat{u_0}\in L^1(\R^d)$, then 
\be
\sup_{T>0}T^{\frac{1}{p-1}}\left\|{\rm e}^{-T(-\Delta)^{\a/2}}u_0\right\|_\infty\gg 1\label{bl-p}
\ee 
is a sufficient condition of blowup of solution of problem \rf{eq}--\rf{in}.
Condition \rf{bl-p} is equivalent to a large value of the Morrey space norm of $u_0$  
\be
\|u_0\|_{M^{d(p-1)/\a}}\gg 1.\label{bl-M}
\ee

(ii) Moreover, if $1<p<p_{\rm F}$ where $p_{\rm F}=1+\frac\a{d}$ is the so-called Fujita exponent, then each nontrivial nonnegative solution $u\not\equiv 0$ blows up in a finite time.  
\end{proposition}

\begin{proof}
(i) Clearly, for $p>1$, $F(u)\approx cu^p$ as $u\searrow 0$, so $h(u)\approx \frac{1}{c(p-1)}u^{1-p}$, $u\searrow 0$. 
Thus, 
$h^{-1}(z)\approx z^{-\frac1{p-1}}\left(\frac{1}{c(p-1)}\right)^{\frac1{p-1}}$ as $z\to\infty$. 
Therefore, the sufficient condition for blowup \rf{blT} becomes 
\be
T^{\frac1{p-1}}\left|{\rm e}^{T{\mathcal A}}u_0(0)\right|\gg 1.\label{bl-a}
\ee
By the translational invariance of equation \rf{eq}, condition \rf{bl-a} for $u_0\ge 0$ is equivalent to 
$\sup_{T>0}T^{\frac{1}{p-1}}\left\|{\rm e}^{T{\mathcal A}}u_0\right\|_{\infty}\gg 1.$   

On the other hand, according to \cite[Theorem 1]{CCR}, for $u_0\in L^1(\R^d)$ with $\widehat{u_0}\in L^1(\R^d)$ and large $t>0$ the semigroup ${\rm e}^{t{\mathcal A}}$ applied to $u_0$ can be well approximated by  ${\rm e}^{-t(-\Delta)^{\a/2}}$ generated by $-(-\Delta)^{\a/2}$  
\be
\lim_{t\to\infty}t^{\frac{d}{\a}}\left\|{\rm e}^{t{\mathcal A}}u_0-{\rm e}^{-t(-\Delta)^{\a/2}}u_0\right\|_{\infty}=0\label{approx}
\ee
while $\left\|{\rm e}^{-t(-\Delta)^{\a/2}}u_0\right\|_{\infty}={\mathcal O}\left(t^{-\frac{d}{\a}}\right)$. 

Next, we have the equivalence 
\be
\sup_{t>0}t^\g\left\|{\rm e}^{-t(-\Delta)^{\a/2}}u\right\|_\infty<\infty\ {\rm if\ and\ only\ if\ } u\in B^{-\g\a}_{\infty,\infty}(\R^d)\label{Bes}
\ee
where $B^{-\kappa}_{\infty,\infty}$ is the homogeneous Besov space of order $-\kappa<0$. The above condition \rf{Bes} is for $u\ge 0$ equivalent to $u\in M^{\frac{d}{\a\g}}(\R^d)$, the Morrey space of order $\frac{d}{\a\g}$,  \cite[Prop. 2B)]{Lem} for $\a=2$ and a slight modification of  \cite[Sec. 4, proof of Prop. 2]{Lem2} for $\a\in(0,2)$.

Finally, condition \rf{bl-a} is equivalent to 
\be
T^{\frac1{p-1}}\left|{\rm e}^{-T(-\Delta)^{\a/2}}u_0(0)\right|\gg 1.\label{bl-aa}
\ee
which, in turn, is equivalent  for $u_0\ge 0$ to condition \rf{bl-M} 
 by above remarks. 

Note that for ${\rm e}^{t{\mathcal A}}={\rm e}^{-t(-\Delta)^{\a/2}}$ the assumptions on the initial data $u_0\in L^1(\R^d)$ with $\widehat{u_0}\in L^1(\R^d)$ can be relaxed to $u_0\in L^1(\R^d)\cap L^\infty(\R^d)$. 

Of course, condition \rf{bl-aa} is quite general, involves one free parameter $T>0$, and 
particular examples of initial data in \cite[Th. 2.3]{A} leading to blowup of solutions satisfy \rf{bl-aa}. 
\medskip 

(ii) Rewriting the quantity in \rf{bl-aa} as 
\bea
T^{\frac{1}{p-1}}\left|{\rm e}^{-T(-\Delta)^{\a/2}}u_0(0)\right|
 &=&  T^{\frac1{p-1}-\frac{d}{\a}} \int R \left(\frac{|x|}{T^\frac{1}{\a}}\right)u_0(x)\dx
\nonumber\\
&\approx&R(0) T^{\frac1{p-1}-\frac{d}{\a}}\|u_0\|_{1}\to\infty,\ \ \ {\rm for}\ \ T\to\infty,\nonumber
\eea 
we see that for each $p<1+\frac\a{d}$ and $\|u_0\|_{1}>0$,  the upper bound equals $\infty$ as claimed; remember relation \rf{approx}. 
Above, the kernel of the semigroup ${\rm e}^{-t(-\Delta)^{\a/2}}$ has selfsimilar form, and is given by $P_{t,\a}(x,t)=t^{-\frac{d}{\a}}R\left(\frac{|x|}{T^\frac{1}{\a}}\right)$ with a smooth function $R$. This satisfies the bound   
\be
0<P_{t,\a}(x)\le \frac{C}{\left( {t}^{1/\a}+|x|\right)^d},\label{GW}
\ee
and, moreover, its gradient satisfies  
\be
|\nabla P_{t,\a}(x)|\le \frac{C}{\left( {t}^{1/\a}+|x|\right)^{d+1}},\label{grGW}
\ee
following from standard estimates for the kernels of linear fractional heat equations, cf. \cite{CCR,Lem2}. 

The proof of (ii) for $p=p_{\rm F}$ and ${\rm e}^{t{\mathcal A}}={\rm e}^{-t(-\Delta)^{\a/2}}$, $\a\in(0,2)$, is in \cite{Su}. 
A rather simple new proof of the result (ii) for $\a=2$ and $p=p_{\rm F}$  is in \cite{B-bl}. 
\end{proof}


 These are counterparts of results in \cite[Remark 7, \, Theorem 2]{B-bl} for the classical nonlinear heat equation. 
These, together with results of \cite[Proposition 2.3]{BP}, lead to the following partial \, {\em dichotomy } result, similarly as was in \cite[Corollary 11]{B-bl} for the Cauchy problem \rf{nlh-a}--\rf{ini}. For analogous questions for radial solutions of chemotaxis systems, see also \cite{B-book}.

\begin{corollary}[dichotomy]\label{dich} 
There exist two positive constants $c(\a,d,p)$ and $C(\a,d,p)$ such that if $p>1+\frac{\a}{d}$   then 
\begin{itemize}
\item[(i)]
 $|\!\!| u_0|\!\!|_{M^{d(p-1)/\a}_q}<c(\a,d,p)$ for some $q\in\left(1,\frac{d(p-1)}{\a}\right)$, implies that problem \rf{nlh-a}--\rf{ini} has a~global in time, smooth  solution satisfying the time decay estimate 
 $\|u(t)\|_\infty={\mathcal O}\left(t^{-1/(p-1)}\right)$.  

\item[(ii)]
  $|\!\!| u_0|\!\!|_{M^{d(p-1)/\a}}>C(\a,d,p)$ implies that each nonnegative solution of problem \rf{nlh-a}--\rf{ini} blows up in a finite time.
\end{itemize}   
\end{corollary}

It is of interest to estimate the discrepancy of  these constants  $c(\a,d,p)$ and $C(\a,d,p)$ compared to  the Morrey space norm 
$$|\!\!| u_\infty|\!\!|_{M^{d(p-1)/\a}_q}=\left(\frac{\s_d}{d-\frac{\a}{p-1}}\right)^{1/q}  s(\a,d,p)$$ of the singular stationary solution $u_\infty>0$ of \rf{nlh-a} which exists for $p>1+\frac{\a}{d-\a}$; here 
\be
\s_d=\frac{2\pi^{d/2}}{\Gamma\left(\frac{d}{2}\right)}\label{pole}
\ee
 is the area of the unit sphere ${\mathbb S}^{d-1}$ in ${\mathbb R}^d$. 
This singular stationary solution $u_\infty>0$ is homogeneous, see  \cite[Prop. 2.1]{BP}, 
\be
 u_\infty(x)=s(\a,d,p)|x|^{-\frac{\a}{p-1}}\label{uC}
\ee 
with  the constant 
\be
s(\a,d,p)= \left(
\frac{2^\a}{\Gamma\left(\frac{\a}{2(p-1)}\right)} 
 \frac{\Gamma\left(\frac{d}{2}-\frac{\a}{2(p-1)}\right) \Gamma\left(\frac{p\a}{2(p-1)}\right)}{\Gamma\left(\frac{d}{2}-\frac{p\a}{2(p-1)}\right)} 
\right)^{\frac{1}{p-1}}.\label{sC}
\ee
Note that asymptotically 
\be
s(\a,d,p)\approx c_{\a,p}d^{\frac{\a}{2(p-1)}}\ \ \ \ {\rm as\ \ \ }d\to\infty \label{s-as}
\ee
with constants $c_{\a,p}$ independent of $d$. 

Of course, there are many interesting behaviors of solutions (and still open questions) for the initial data of intermediate size satisfying 
$$
c(\a,d,p)\le |\!\!| u_0|\!\!|_{M^{d(p-1)/\a}}\le C(\a,d,p),
$$ 
and/or suitable pointwise estimates comparing the initial condition  $u_0$ with the singular solution $u_\infty$, see e.g. \cite{BP}. 
One of the results in this direction is \cite[Theorem 2.6]{BP}: 
if  $\a\in(0,2)$, $p>1+\frac{\a}{d-\a}(>1+\frac{\a}{d})$, $u=u(x,t)$ is a solution of problem \rf{nlh-a}--\rf{ini}   with   $0\le  u_0(x)\le  u_\infty(x)$
(plus some qualitative assumptions like $u_0\in M^{d(p-1)/\alpha}(\R^d)\cap M^{\ts}(\R^d)$ with $\ts>d(p-1)/\alpha$, $\lim_{x\to 0}|x|^\ap u(x,t)=\lim_{x\to \infty}|x|^\ap u(x,t)=0$, uniformly\ in\ \  $t\in(0,T)$), 
then $u$  can be continued to a global in time solution which still satisfies the bound $0\le u(x,t)\le  u_\infty(x)$.  
 
This is a natural extension of properties of the Cauchy problem \rf{nlh}, \rf{ini} studied in, e.g.,  \cite{QS,S} and \cite{B-bl}.

\section{Estimates of discrepancy}  
Similarly to the considerations in \cite{BZ-2} on blowup for radial solutions of chemotaxis systems, we determine asymptotic (with respect to the variable of  dimension $d\to\infty$) discrepancy between bounds in sufficient conditions for blowup either  in terms of multiple of the singular solution or in terms of critical value of the radial concentration and therefore the Morrey norm of the initial data for the model problem \rf{nlh-a}--\rf{ini}.

\begin{theorem}
(i) For each $\alpha\in(0,2]$ and $p>1+\frac{\a}{d}$ there exists a constant $\nu_{\a,p}$ independent of the dimension $d$ such that if $N>\nu_{\a,p}$, then each solution of the Cauchy problem \rf{nlh-a}--\rf{ini} in $\R^d$ with the initial data $u_0(x)\ge N u_\infty(x)$ blows up in a finite time. 

(ii) For $\alpha=2$ and $p>1+\frac{2}{d}$ there exists a constant $\kappa_{2,p}$ independent of $d$ such that if the $\frac{d(p-1)}{2}$-radial concentration of $u_0\ge 0$ defined by 
\be
\xn u_0\xn_\frac{d(p-1)}{2}\equiv \sup_{r>0}r^{\frac{2}{p-1}-d}\int_{\{|y|<r\}}u_0(y)\dy\label{2-conc}
\ee 
satisfies 
\be
\xn u_0\xn_\frac{d(p-1)}{2}\ge\kappa \s_d d^{1/2(p-1)}
\label{su-2}
\ee
 with some $\kappa>\kappa_{2,p}$ then each solution of \rf{nlh-a}--\rf{ini} in $\R^d$  blows up in a finite time. 

(iii) For $\alpha\in(0,2)$ and $p>1+\frac{\a}{d}$( $p>1+\frac{\a}{d-2}$ so that $d\ge 3$) there exists a constant $\kappa_{\a,p}$ independent of $d$ such that if the $\frac{d(p-1)}{\a}$-radial concentration of $u_0\ge 0$ defined by 
\be
\xn u_0\xn_\frac{d(p-1)}{\a}\equiv \sup_{r>0}r^{\frac{\a}{p-1}-d}\int_{\{|y|<r\}}u_0(y)\dy\label{a-conc}
\ee 
satisfies 
\be
\xn u_0\xn_\frac{d(p-1)}{\a}\ge\kappa\s_d d^{\a/2(p-1)}\label{su-a}
\ee
 with $\kappa>\kappa_{\a,p}$ then each solution of \rf{nlh-a}--\rf{ini} in $\R^d$  blows up in a finite time. 
\end{theorem}

\begin{remark}
The $\frac{d(p-1)}{\a}$-radial concentration defined in \rf{a-conc} is comparable with this Morrey norm in $M^{{d(p-1)}/{\a}}(\R^d)$ 
\be
c_d|\!\!|u_0|\!\!|_{M^{d(p-1)/{\a}}}\le 
\sup_{r>0}r^{\frac{\a}{p-1}-d}\int_{\{|y|<r\}}u_0(y)\dy\le |\!\!|u_0|\!\!|_{M^{d(p-1)/{\a}}}. 
\ee 
However,   the comparison constant $c_d$ depends on $d$, cf. \cite[Proposition 7.1]{BKZ2}).  
\end{remark}

\begin{proof}
A more detailed analysis of condition \rf{bl-a} reveals that 
\be
\sup_{t>0}t^{\frac{1}{p-1}}{\rm e}^{-t(-\Delta)^{\a/2}}u_0(0)>c_{\a,p}\label{suff-a}
\ee
is a sufficient condition for blowup, with some constant $c_{\a,p}>0$ {\em independent\,} of $d$. For $\a=2$ we simply have $c_{2,p}=\left(\frac{1}{p-1}\right)^{\frac{1}{p-1}}$, see \cite{B-bl}.

First, we compute 
\bea
K_{2,p}(d)&\equiv&\sup_{t>0}t^{\frac{1}{p-1}}{\rm e}^{t\Delta}(u_\infty)(0)\nonumber\\
&=&s(2,d,p)\sup_{t>0}t^{\frac{1}{p-1}-\frac{d}{2}}\s_d\int_0^\infty {\rm e}^{-r^2/4t}(4\pi)^{-d/2}r^{-\frac{2}{p-1}+d-1}\dr\nonumber\\
&=& s(2,d,p)4^{-\frac{1}{p-1}}\frac{2}{\Gamma\left(\frac{d}{2}\right)} \int_0^\infty{\rm e}^{-\tau}\tau^{\frac{d-1}{2}-\frac{1}{p-1}-\frac12}\dta\nonumber\\
&=&s(2,d,p)2^{1-\frac{2}{p-1}}\frac{\Gamma\left(\frac{d}{2}-\frac{1}{p-1}\right)}{\Gamma\left(\frac{d}{2}\right)}\nonumber\\
&\asymp&d^{\frac{1}{p-1}-\frac{1}{p-1}}\asymp 1.\label{k-2}
\eea
The last two lines follow using the relation 
\be
\frac{\Gamma\left(z+a\right)}{\Gamma\left(z+b\right)}\asymp z^{a-b}\ \ {\rm as\ } z\to\infty,\label{G-a-b}
\ee
see \cite{TE}, which is an immediate consequence of the Stirling formula 
\be
\Gamma(z+1)\approx \sqrt{2\pi z}\, z^z{\rm e}^{-z}\ \ {\rm as\ \ }z\to \infty.\label{Stir}
\ee

Next, for $\a\in(0,2)$ we need a representation of the kernel of the semigroup ${\rm e}^{-t(-\Delta)^{\a/2}}$ using  the Bochner subordination formula, cf. \cite[Ch. IX.11]{Y}
\be
{\rm e}^{-t(-\Delta)^{\alpha/2}}=\int_0^\infty f_{t,\alpha}(\lambda){\rm e}^{\lambda\Delta}\dl, \ \ {\rm or}\ \ P_{t,\a}(x)=\int_0^\infty f_{t,\alpha}(\lambda)P_{t,2}(x)\dl,\label{subord}
\ee
with some functions $f_{t,\alpha}(\lambda)\ge 0$ {\em independent } of $d$. In fact, the subordinators $f_{t,\alpha}$ satisfy 
$ 
{\rm e}^{-ta^\alpha}=\int_0^\infty f_{t,\alpha}(\lambda){\rm e}^{-\lambda a}\dl,
$ 
so that they have selfsimilar form $f_{t,\alpha}(\lambda)=t^{-\frac{1}{\alpha}}f_{1,\alpha}\left(\lambda t^{-\frac{1}{\alpha}}\right)$.

Then, we have extensions of the previous computations 
\bea
K_{\a,p}(d)&=&\sup_{t>0}t^{\frac{1}{p-1}}{\rm e}^{-t(-\Delta)^{\a/2}}\left(\frac{s(\a,d,p)}{|x|^{\frac{a}{p-1}}}\right)(0)\nonumber\\
&=&s(\a,d,p)\sup_{t>0}t^{\frac{1}{p-1}-\frac{d}{\a}}\s_d\int_0^\infty R\left(\frac{r}{t^{1/\a}}\right)r^{-\frac{\a}{p-1}+d-1}\dr\nonumber\\
&=&s(\a,d,p)\s_d\int_0^\infty R(\r)\r^{d-1-\frac{\a}{p-1}}\,{\rm d}\r \nonumber\\ 
&=&s(\a,d,p)\s_d\int_0^\infty\int_0^\infty f_{1,\a}(\lambda)(4\pi)^{-d/2}\lambda^{-d/2}{\rm e}^{-\r^2/4\lambda}\r^{d-1-\frac{\a}{p-1}}\dl \,{\rm d}\r \nonumber\\
 &\approx&s(\a,d,p)\frac{1}{\Gamma\left(\frac{d}{2}\right)}4^{-\frac{\a}{2(p-1)}}\int_0^\infty f_{1,\a}(\lambda)\lambda^{-\frac{\a}{2(p-1)}}\dl \times \int_0^\infty {\rm e}^{-\tau}\tau^{\frac{d}{2}-1-\frac{\a}{2(p-1)}}\dta\nonumber\\
 &\asymp& s(\a,d,p) \frac{\Gamma\left(\frac{d-\frac{\a}{p-1}}{2}\right)}{\Gamma\left(\frac{d}{2}\right)}\nonumber\\
&\asymp& d^{\frac{\a}{2(p-1)}-\frac{\a}{2(p-1)}}\asymp 1\label{k-a}
\eea
by representation \rf{subord},  relations \rf{s-as} and \rf{G-a-b}. 

Using relations \rf{k-2}, \rf{k-a} and the comparison principle, we see that 
$$
N\kappa_{\a,p}(d)>c_{\a,p}$$
 suffices to a finite time blowup, thus (i) follows since the bound for $K_{\a,p}(d)$ is $d$-independent.

(ii) Now, let us compute the asymptotics of the expression in condition \rf{bl-a} for $\a=2$ and  the normalized Lebesgue measure  ${\rm d}S$ on the unit sphere ${\mathbb S}^{d-1}$ 
\bea
L_{2,p}(d)&=&\sup_{t>0}t^{\frac{1}{p-1}}{\rm e}^{-t\Delta}\left(\frac{{\rm d}S}{\s_d}\right)\nonumber\\
&=&\sup_{t>0}(4\pi)^{-d/2}t^{\frac{1}{p-1}-\frac{d}{2}}{\rm e}^{-1/4t}\nonumber\\
&=&4^{-\frac{1}{p-1}}\pi^{-d/2} \left(\frac{1}{\rm e}\left(\frac{d}{2}-\frac{1}{p-1}\right)\right)^{\frac{d}{2}-\frac{1}{p-1}}
\nonumber\\
 &\approx&4^{-\frac{1}{p-1}}\frac{1}{\s_d}\frac{1}{\Gamma\left(\frac{d}{2}\right)} \frac{\Gamma\left(\frac{d}{2}-\frac{1}{p-1}+1\right)}{\left(2\pi\left(\frac{d}{2}-\frac{1}{p-1}\right)\right)^{1/2}}
\nonumber\\
&\asymp&\frac{1}{\s_d}d^{\frac{1}{2}-\frac{1}{p-1}}. \label{as-2-L}
\eea
Indeed, for $\frac{d}{2}>\frac{1}{p-1}$ the quantity $\max_{t>0}t^{\frac{1}{p-1}-\frac{d}{2}}{\rm e}^{-1/4t}$ is attained for $t_0=\left(4\left(\frac{d}{2}-\frac{1}{p-1}\right)\right)^{-1}$, and then relation \rf{G-a-b} is used.

(iii) 
For $\a\in(0,2)$, instead of \rf{su-2}  an analogous sufficient condition is of different order than for $\a=2$, namely \rf{su-a}
\be
\xn u_0\xn_{\frac{d(p-1)}{\a}}= \sup_{r>0}r^{\frac{\a}{p-1}-d}\int_{\{|y|<r\}}u_0(y)\dy\ge \kappa_{\a,p}\s_dd^{\a/2(p-1)}\label{su-a}
\ee 
with a constant $\kappa_{\a,p}$ {\em independent} of $d$. 
Indeed, for the normalized Lebesgue measure  ${\rm d}S$ on the unit sphere ${\mathbb S}^{d-1}$ we have an upper bound for the quantity 
\bea
L_{\a,p}(d)&=&\sup_{t>0}t^{\frac{1}{p-1}}{\rm e}^{-t(-\Delta)^{\a/2}}\left(\frac{{\rm d}S}{\s_d}\right)\nonumber\\
&=&\sup_{t>0}t^{\frac{1}{p-1}-\frac{d}{\a}}
R\left(\frac{1}{t^{1/\a}}\right)\nonumber\\
&=&\sup_{\r>0}\r^{d-\frac{\a}{p-1}}R(\r)\nonumber\\
&=&\sup_{\r>0}\int_0^\infty f_{1,\a}(\lambda)(4\pi\lambda)^{-d/2}\r^{d-\frac{\a}{p-1}}{\rm e}^{-\r^2/4\lambda}\dl\label{aaa}\\
&\le& 4^{-\frac{\a}{2(p-1)}}\frac{2}{\s_d\Gamma\left(\frac{d}{2}\right)}\int_0^\infty \sup_{\r>0}\left( f_{1,\a}\left(\frac{\r^2}{4\tau}\right)\left(\frac{\r^2}{4\tau}\right)^{1-\frac{\a}{2(p-1)}}\right)
\tau^{\frac{d}{2}-\frac{\a}{2(p-1)}-1}{\rm e}^{-\tau}\dta\nonumber\\
 &\asymp&\frac{1}{\s_d\Gamma\left(\frac{d}{2}\right)}\Gamma\left(\frac{d-\frac{a}{p-1}}{2}\right)\nonumber\\
&\asymp&\frac{1}{\s_d}d^{-\frac{\a}{2(p-1)}}\nonumber
\eea 
since representation \rf{subord}, formulas \rf{pole} and \rf{G-a-b} hold.   

For an asymptotic lower bound on the quantity $L_{\a,p}(d)$, begin with the observation that for $\beta=\frac{d}{2}-\frac{\a}{2(p-1)}-1$ ($\beta>0$ since $p>1+\frac{\a}{d-2}$)
\bea
m\equiv\max_{\tau>0}{\rm e}^{-\tau}\tau^{\beta} 
&=& {\rm e}^{-\tau_0}\tau_0^{\beta}\ \ \ \ {\rm with}\ \ \ \  \tau_0= \beta  \nonumber\\ 
&=&{\rm e}^{-\beta}\beta^\beta\nonumber\\
&\approx&\Gamma\left(\beta+1\right)\frac{1}{\sqrt{2\pi\beta}}\label{min}
\eea
holds by \rf{Stir}. 
Now, let $h\asymp d^{\frac12}$. 
It is easy to check that 
 $$\frac{1}{m}\min_{[\tau_0,\tau_0+h]}{\rm e}^{-\tau}\tau^{\beta} \ge \eta$$ 
 for some $\eta>0$, uniformly in $d$. 
Indeed, 
$$
\log\frac{(d+h)^d{\rm e }^{-d-h}}{d^d{\rm e}^{-d} }=d\log\left(1+\frac{h}{d}\right)-h\approx d\frac{h}{d}-\frac{dh^2}{2d^2}-h= {\mathcal O}\left(\frac{h^2}{2d}\right).
$$
From formulas \rf{aaa} and \rf{min} we infer 
\bea 
L_{\alpha,p}(d)&\ge& 4^{-\frac{\a}{2(p-1)}}\frac{2}{\s_d\Gamma\left(\frac{d}{2}\right)}
\sup_{\r>0} \int_{\tau_0}^{\tau_0+h} \left( f_{1,\alpha}\left(\frac{\r^2}{4\tau}\right)\left(\frac{\r^2}{4\tau}\right)^{1-\frac\alpha{2(p-1)}}\right)\tau^{\frac{d}{2}-\frac{\alpha}{2(p-1)}-1}{\rm e}^{-\tau}\,{\rm d}\tau \nonumber\\
&\ge&4^{-\frac{\a}{2(p-1)}}\frac{2}{\s_d\Gamma\left(\frac{d}{2}\right)}\frac{\eta h}{\sqrt{d}}\Gamma\left(\frac{d}{2}-\frac{\a}{2(p-1)}\right)\nonumber\\
&\asymp&\frac{1}{\s_d}d^{-\frac{\a}{2(p-1)}}\nonumber
\eea 
Therefore $L_{\alpha,p}(d)\asymp  \frac{1}{\s_d}d^{-\frac\alpha{2(p-1)}}$ 
holds.
This is an estimate of optimal order and {\em different} from its counterpart  for $\alpha=2$. 
Now, it is clear that a sufficient condition for blowup is satisfied if 
$$NL_{\a,p}(d)>c_{\a,p}$$
 with either $N=\kappa\s_dd^{1/(p-1)}$ if $\a=2$ or $N=\kappa\s_dd^{\a/2(p-1)}$ if $\a\in(0,2)$. 
\end{proof}

\end{document}